\documentclass[12pt, reqno]{amsart}

\usepackage{amsmath,amsfonts,amsthm,amssymb,color,hyperref}
\usepackage{xcolor}
\usepackage{mathrsfs}

\usepackage[T1]{fontenc}

\usepackage{graphicx}
\usepackage{subfigure} 

\makeatletter
     \def\section{\@startsection{section}{1}%
     \z@{.7\linespacing\@plus\linespacing}{.5\linespacing}%
     {\bfseries
     \centering
     }}
     \def\@secnumfont{\bfseries}
     \makeatother

\usepackage{graphicx}

\newcommand{\R}{\mathbb R}
\newcommand{\RR}{\mathbb R}
\newcommand{\N}{\mathbb N}

\newcommand{\E}{\mathbb E}

\newcommand{\HH}{\mathfrak H}

\newtheorem{theorem}{Theorem}[section]
\newtheorem{lemma}[theorem]{Lemma}
\newtheorem{proposition}[theorem]{Proposition}
\newtheorem{corollary}[theorem]{Corollary}
\theoremstyle{definition}

\theoremstyle{remark}
\newtheorem{remark}{Remark}
\numberwithin{equation}{section}
\begin{document}
\title[Gaussian fluctuations for the stochastic heat equation]{Gaussian fluctuations for the stochastic heat equation with colored noise}

 \author[J. Huang]{Jingyu Huang}
\address{University of Birmingham, School of Mathematics, UK}
\email{j.huang.4@bham.ac.uk}

\author[D. Nualart]{David Nualart} \thanks {D. Nualart is supported by NSF Grant DMS 1811181.}
\address{University of Kansas, Department of Mathematics, USA}
\email{nualart@ku.edu}

\author[L. Viitasaari]{Lauri Viitasaari}
\address{University of Helsinki, Department of Mathematics and Statistics, Finland}
\email{lauri.viitasaari@iki.fi}

\author[G. Zheng]{Guangqu Zheng}
\address{University of Kansas, Department of Mathematics, USA}
\email{zhengguangqu@gmail.com}

\begin{abstract}
In this paper, we present a {\it quantitative} central limit theorem for  the     $d$-dimensional stochastic heat  equation driven by a   {\it Gaussian  multiplicative  noise}, which is {\it white} in time and  has  a spatial covariance given by  the {\it Riesz kernel}.   We   show that the   spatial average of the solution over an  Euclidean ball  is close to a Gaussian distribution, when the radius of the ball tends to infinity. Our central limit theorem is  described in the {\it  total variation distance}, using Malliavin calculus and Stein's method. We also provide a functional central limit theorem.
\end{abstract}

\maketitle



\allowdisplaybreaks

\section{Introduction}

 We consider the   stochastic heat equation  
\begin{equation}
\label{eq:heat-equation}
\frac{\partial u }{\partial t}  = \frac{1}{2}   \Delta u + \sigma(u) \dot{W} ,
\end{equation}
on $\R_+\times \R^d$ with initial condition $u(0,x)=1$, where $\dot{W}(t,x)$ is a centered Gaussian noise with  covariance
$
\E\big[ \dot{W}(t,x) \dot{W}(s,y)\big] = \delta_0(t-s) | x-y|^{-\beta}$, $ 0<\beta < \min(d,2)$.
In this paper,  $\vert \cdot\vert$ denotes the Euclidean norm and 
  we assume    the nonlinear term
  $\sigma$ is  a Lipschitz function with $\sigma(1)\neq 0$; see point (iii) in Remark \ref{rem0}.
  
 \medskip
    Our first result is the following quantitative central limit theorem concerning the spatial average of the solution $u(t,x)$ over $B_R:= \{ x\in \R^d: |x| \le R\}$.
\begin{theorem} \label{thm:TV-distance}
For all $ t > 0$,  there exists a constant $C = C(t,\beta)$, depending on $t$ and $\beta$, such that 
\[
d_{\rm TV}\left(   \frac{1}{\sigma_R} \int_ {B_R} \big[ u(t,x) - 1 \big]\,dx, ~Z\right) \leq C R^{-\beta/2} \,,
\]
where $d_{\rm TV}$ denotes total variation distance {\rm (see \eqref{TV_def})},  $Z\sim N(0,1)$ is  a standard normal random variable and
$\sigma_R^2 = {\rm Var} \big( \int_ {B_R} [ u(t,x) - 1 ]\,dx \big)$. Moreover, the normalization   $\sigma_R$ is of order $R^{d-\frac{\beta}{2}}$, as $R\rightarrow +\infty$ {\rm (see \eqref{n0})}.
\end{theorem}

\begin{remark}
In the above result we have assumed $u(0,x)=1$ for the sake of simplicity. However, one can easily extend the result to cover more general initial deterministic condition $u(0,x)$. In fact, this is the topic of Corollary \ref{cor:TV-general-initial}.
\end{remark}

 We will mainly rely on the methodology of Malliavin-Stein approach to prove the above result.  Such an approach   was    introduced by Nourdin and Peccati in \cite{NP09} to, among other things,  quantify Nualart and Peccati's fourth moment theorem in \cite{NP05}. Notably, if $F$ is a Malliavin differentiable
  (that means, $F$ belongs to the Sobolev space $\mathbb{D}^{1,2}$), centered Gaussian functional  with unit variance,   the   well-known  Malliavin-Stein bound implies
 \begin{align}
 d_{\rm TV}(F, Z) : & = \sup\Big\{ P(F\in A) -  P(Z\in A) \,:\, A \subset \R  \,\,\,\, \text{Borel sets} \Big\}  \label{TV_def}  \\
 &\leq 2 \sqrt{ {\rm Var}\big(   \langle DF, -DL^{-1}F  \rangle_{\HH} \big) } \quad \text{with $Z\sim N(0,1)$,} \label{eq:dist Malliavin}
 \end{align}
where $D$ is the Malliavin derivative, $L^{-1}$ is the pseudo-inverse of the Ornstein-Uhlenbeck operator
and $\langle \cdot , \cdot \rangle_\HH$ denotes the inner product in the Hilbert space $\HH$ associated with the covariance of $W$; see also the monograph \cite{NP}. 

\medskip

It was   observed in \cite{Zhou} that   instead of  $ \langle DF, -DL^{-1}F  \rangle_\HH$, one can work with the term $ \langle DF, v  \rangle_\HH$ once $F$  can be represented as a Skorohod integral $\delta(v)$; see the following result from   \cite[Proposition 3.1]{Zhou} (see also \cite{HNV18,Eulalia}).

\begin{proposition}  \label{lem: dist}
If $F$ is a centered  random  variable in the Sobolev space $\mathbb{D}^{1,2}$ with unit variance such that  $F = \delta(v)$ for  some $\HH$-valued random variable $v$,     then,  with $Z \sim N (0,1)$,
\begin{equation}\label{eq:dist}
d_{\rm TV}(F, Z) \leq 2 \sqrt{{\rm Var} \big(  \langle DF, v\rangle_\HH  \big) }.
\end{equation}

 {\rm   This estimate enables us to bring in   tools from stochastic analysis. }
\end{proposition}

\begin{remark}
Throughout this paper we only work with the total variation distance. However, we point out that we can get the same rate in other  frequently used distances. Indeed, if
$\texttt{d}$ denotes either the   Kolmogorov distance, the Wasserstein distance or the  Fortet-Mourier distance, we have as well
\[
\texttt{d}(F,Z) \leq 2 \sqrt{{\rm Var} \big(  \langle DF, v\rangle_\HH  \big) } ,
\]
where $F,Z$ are given  as in  \eqref{eq:dist}; see \emph{e.g.} \cite[Chapter 3]{NP} for the properties of Stein's solution.  Thus, proceeding in the exact same lines as in this paper, we will get the same rate in these distances.  
\end{remark}

It is known (see, \emph{e.g.} \cite{Dalang99, Walsh})
that  equation \eqref{eq:heat-equation} has a unique  \emph{mild solution} $\{u(t,x): t\geq0, x \in \R^d\}$, in the sense that it is adapted to the filtration generated by $W$,  uniformly bounded in $L^2(\Omega)$  over $ [0,T]\times  \R^d$ (for  any finite $T$) and satisfies  the following integral equation in the sense of Dalang-Walsh
\begin{equation}\label{eq: mild}
u(t,x) = { \int_{\RR^d}p_t(x-y)u(0,y)dy}+   \int_0^t \int_{\RR^d}  p_{t-s}(x-y)  \sigma(u(s,y))W(ds,dy)\,, 
\end{equation}
where   $p_t(x) = (2\pi t)^{-d/2}\exp\big( - |x|^2/ (2t) \big)$ denotes the   heat kernel and is the {\it fundamental solution} to the corresponding {\it deterministic} heat equation. 

\medskip

Let us  introduce some handy notation. For fixed $t > 0$, we define
 \begin{align}\label{NOTA1}
 \varphi_R(s,y) :=  \int_{B_R}   p_{t-s} (x-y) dx  \,\,\,\,  \text{and} \,\,\,\, G_R(t): = \int_{B_R} \big[ u(t,x) - 1 \big] dx \,.
 \end{align}
If we put $F_R(t) := G_R(t)/\sigma_R$, then we   write, due to  the  Fubini's theorem,  
\begin{align}\label{FRT}
F_R(t) =  \frac{1}{\sigma_R}\int_0^t \int_{\R^d} \varphi_R(s,y) \sigma (u(s,y)) W(ds,dy) = \delta (v_R  )\,,
\end{align}
with $v_R(s,y) = \sigma_R^{-1}\mathbf{1} _{[0,t]}(s) \varphi_R(s,y) \sigma  (u(s,y) )$ taking into account  
 that  the Dalang-Walsh integral \eqref{eq: mild} is a particular case of the Skorohod integral; {see \cite{Nualart}. }

\medskip

By Proposition  \ref{lem: dist},   the proof  of Theorem \ref{thm:TV-distance}  reduces to  estimating the variance of   $ \langle DF_R(t), v_R\rangle _\HH$. One of the key steps, our Proposition \ref{pro:covariance2},  provides the exact asymptotic behavior of the normalization $\sigma_R^2$: 
\begin{align}\label{n0}
\sigma_R^2 = {\rm Var}\big( G_R(t) \big)  \sim  \left( k_\beta \int_0^t \eta^2(s)\, ds\right) R^{2d-\beta}\,,  \quad  \text{as  $R\to+\infty$;}
\end{align}
and throughout this paper, we will reserve the notation $k_\beta$ and $\eta(\cdot)$ for 
\begin{align}\label{NOTA2}
k_\beta := \int_{B_1^2} |x_1-x_2 | ^{-\beta} dx_1dx_2 \quad{\rm and}\quad  \eta(s) := \E\big[ \sigma( u(s,y) ) \big] \,.   
\end{align}

\begin{remark}\label{rem0} (i) The definition of $\eta$ does not depend on the spatial variable, due to the \emph{strict stationarity} of the solution, meaning  that the finite-dimensional distributions of the process $\{u(t, x+y),x\in \RR^d\}$ do not depend on $y$; see \cite{Dalang99}.

\medskip

(ii) Another consequence of the \emph{strict stationarity}   is that the quantity
\begin{equation}\label{Kp}
K_p(t):=\E\big[  |u(t,x)| ^p \big]
\end{equation}
does not depend on $x$. Moreover,   $K_p(\cdot)$ is uniformly bounded on compact sets. 

\medskip

(iii) Under  our constant initial condition (\emph{i.e.} $u(0,x)\equiv 1$), the   assumption $\sigma(1) \neq 0$ is  {\it necessary and  sufficient} in our paper. It is necessary, in view of the usual \emph{Picard iteration}, to exclude the situation where $u(t,x)\equiv 1$ is the unique solution, and it is sufficient to  guarantee that the integral in \eqref{n0} is nonzero. Moreover, we have the following equivalence
\begin{center} 
$\sigma(1)=0 \Leftrightarrow \sigma_R=0,\forall R>0 \Leftrightarrow  \sigma_R=0$ for some  $R>0 \Leftrightarrow{\displaystyle  \lim_{R\to\infty}\sigma_R^2 R^{\beta-2d}= 0;}$
 \end{center}
whose verification can be done in the same way as in \cite[Lemma 3.4]{DNZ18}, so we leave it as an easy exercise for interested readers.

 \medskip

(iv) Due to the stationarity again, we can see that Theorem \ref{thm:TV-distance} still holds true when we replace  $B_R$ by  $ \{ x\in\R^d:  | x - a_R | \leq R \}$, with $a_R$ possibly varying as $R\to+\infty$. Moreover, the normalization $\sigma_R$ in this translated version remains unchanged. To see this translational ``invariance'', one can alternatively go through the exactly same arguments as in the proof of Theorem \ref{thm:TV-distance}. Note that this invariance under translation justifies the statement in our abstract. 

\medskip

(v) By going through the exactly same arguments as in the proof of Theorem \ref{thm:TV-distance}, we can obtain the following result:  If we replace the ball $B_R$ by the box $\Lambda_R: = \big\{  (x_1,\ldots, x_d)\in\R^d: \vert x_i\vert\leq R, i=1,\dots, d \big\}$, Theorem \ref{thm:TV-distance} still holds true with a slightly different normalization $\sigma'_R$ given by 
\[
\sigma'_R \sim     \left(\int_0^t \eta^2(s)\int_{\Lambda_1^2} |z- y | ^{-\beta} dzdy  \, ds\right)^{1/2} R^{d-\frac{\beta}{2}}\,,  \quad  \text{as  $R\to+\infty$.}
\]
\end{remark}

 \medskip

Our second main result is the following    functional version of Theorem \ref{thm:TV-distance}.
 
\begin{theorem}  \label{thm:functional-CLT}
Fix any finite $T>0$ and recall  \eqref{NOTA2}. Then,   as $R\to+\infty$, we have
$$
\left\{  R^{\frac {\beta} 2-d}  \int_{B_R} \big[ u(t,x) -1 \big] \,dx   \right\}_{t\in [0,T]} \Rightarrow \left\{ \sqrt{k_\beta} \int_0^t   \eta(s)  dY_s\right\}_{t\in [0,T]} \,,
$$
where    $Y$ is a standard Brownian motion and the above weak convergence takes place on the space of continuous functions $C([0,T])$.
\end{theorem}

 A similar problem for the stochastic heat equation on $\R$  has been recently  considered in  \cite{HNV18}, but only in the case of a space-time white noise, where $\eta^2(s)$    appearing in the limiting variance \eqref{n0} is replaced by  $ \E\big[ \sigma^2(u(s,y)) \big]$. Such a phenomenon also appeared in the case of the one-dimensional wave equation, see the recent paper \cite{DNZ18}, whose authors also considered the Riesz kernel for the spatial fractional noise therein, which    corresponds to the case $\beta=2-2H$.

\begin{remark}When $\sigma(x)=x$, the  random variable $G_R(t)$ defined in  (\ref{NOTA1})
  has an explicit Wiener chaos expansion:
  \[
  G_R(t) =   \int_0^t\int_{\R^d} \varphi_R(s,y) \, W(ds, dy)  + \text{\it higher-order chaoses.}
  \]
In this linear case, the asymptotic result \eqref{n0} reduces to    $\sigma_R^2 \sim t k_\beta R^{2d-\beta}$, while the first chaotic component of $G_R(t)$ is  centered Gaussian with variance equal to 
\[
  \int_0^t  \int_{\R^{2d}}  \varphi_R(s,y) \varphi_R(s,z) \vert  y -z\vert^{-\beta} \, dydz \sim   t  k_\beta   R^{2d-\beta}   \,,\quad\text{as} \,\, R\rightarrow +\infty;
\]
 see \eqref{quan2}.
Therefore, due to the orthogonality of Wiener chaoses with different orders, we can conclude that $F_R(t) = G_R(t)/\sigma_R$ is asymptotically Gaussian.  It is worth pointing out that, unlike in our case, for the {\it linear} stochastic heat equation driven by  space-time white noise as considered in \cite{HNV18}, the central limit    is {\it chaotic}, meaning that  each projection on the Wiener chaos contributes to the Gaussian limit. In this case, the proof of asymptotic normality could be based on the chaotic central limit theorem from  \cite[Theorem 3]{HN} (see also \cite[Section 6.3]{NP}). 
\end{remark}

The rest of the paper is organized as follows. We present the proofs in    Sections \ref{sec:thm1} and \ref{sec:thm2}, after we recall briefly some preliminaries in Section \ref{sec:prel}. 

  Along the paper, we will denote by $C$ a generic constant that may depend on the fixed time $T$, the parameter $\beta$ and $\sigma$,  and it can vary from line to line. We also denote by $\| \cdot \| _p$  the  $L^p(\Omega)$-norm and by $L$ the Lipschitz constant of $\sigma$.

\section{Preliminaries}
\label{sec:prel}

Let us build an isonormal Gaussian process from the colored noise $W$ as follows. We begin with 
a centered Gaussian family of random variables 
$
\big\{W(\psi):\: \psi\in C_c^{\infty}\left([0,\infty)\times\R^{d}\right)\big\}
$, defined on some complete probability space, such that
\begin{align*}
\E\left[W(\psi)W(\phi)\right] =& \int_0^{\infty} \int_{\R^{2d}}\psi(s,x)\phi(s,y)|x-y|^{-\beta} \, d xd yds =: \big\langle \psi, \phi\big\rangle_\HH,
\end{align*}
where  $C_c^{\infty}\left([0,\infty)\times\R^{d}\right)$ is the space of infinitely differentiable functions with compact support on $[0,\infty)\times\R^{d}$.
Let $\mathfrak{H}$ be the Hilbert space defined as the completion of  $C_c^{\infty}\left([0,\infty)\times\R^{d}\right)$ with respect to the above inner product.
By a density argument, we obtain an isonormal Gaussian process  $W:=\big\{W(\psi):\: \psi\in\mathfrak{H}\big\}$  meaning that  for any $ \psi, \phi \in \HH$,
\[
\E\left[W(\psi)W(\phi)\right]=\big\langle \psi, \phi\big\rangle_\HH.
\]

   Let $\mathbb{F}=(\mathcal{F}_t,t\ge 0)$ be the natural filtration of $W$, with $\mathcal{F}_t$ generated by $\big\{W(\phi):\phi$ continuous and with compact support in  $[0,t]\times \R^d\big\}$. 
            Then for any $\mathbb{F}$-adapted, jointly measurable random field $X$ such that
\begin{equation}\label{eq:integrability cond}
 \E ( \| X \|^2_{\HH})= \E\left( \int_0^\infty d s\int_{\R^{2d}}d x d y\:
X(s,y)X(s,x) |x-y|^{-\beta} \right) <\infty,
\end{equation}
the stochastic integral
\[
\int_0^\infty \int_{\R^d} X(s,y)W(ds, d y)
\]
is well-defined in the sense of Dalang-Walsh and the following isometry \mbox{property} is satisfied
\[
 \E\left( \left| \int_0^\infty \int_{\R^d} X(s,y)W(ds, d y) \right|^2 \right)= 
 \E ( \| X \|^2_{\HH}).
\]

\medskip

In the sequel, we recall some basic facts on Malliavin calculus and we refer   readers to   the books \cite{Nualart, Eulalia} for any unexplained notation and result. 

\medskip

 Denote by $C_p^{\infty}(\R^n)$ the space of smooth functions with all their partial derivatives having at most polynomial growth at infinity. Let $\mathcal{S}$ be the space of simple functionals of the form 
$F = f(W(h_1), \dots, W(h_n))
$ for $f\in C_p^{\infty}(\RR^n)$ and $h_i \in \HH$, $1\leq i \leq n$, $n\in\N$. Then, $DF$ is the $\HH$-valued random variable defined by
$$
DF=\sum_{i=1}^n  \frac {\partial f} {\partial x_i} (W(h_1), \dots, W(h_n)) h_i\,.
$$
 The derivative operator $D$  is   \mbox{closable}   from $L^p(\Omega)$ into $L^p(\Omega;  \HH)$ for any $p \geq1$ and   we let $\mathbb{D}^{1,p}$ be the \mbox{completion} of $\mathcal{S}$ with respect to the norm
$$\|F\|_{1,p} = \Big(\E \big[ |F|^p  \big] +   \E \big[  \|D F\|^p_\mathfrak{H}  \big]   \Big)^{1/p}
.$$
We denote by $\delta$ the adjoint of  $D$ characterized  by the    integration-by-parts \mbox{formula}
$$
\E [\delta(u) F ] = \E\big[ \langle u, DF \rangle_\HH\big]
$$
for any $F \in \mathbb{D}^{1,2}$ and $u\in{\rm Dom} \, \delta \subset L^2(\Omega; \HH)$,  the domain of $\delta$. The operator $\delta$ is
also called the Skorohod integral,   because in the case of the Brownian motion, it coincides
with Skorohod's extension of the It\^o integral (see \emph{e.g.} \cite{GT, NuPa}). 
In our context, any   adapted random field $X$ satisfying  \eqref{eq:integrability cond} belongs to $\text{Dom}\delta$, and $\delta (X)$  coincides with the Dalang-Walsh stochastic integral. As a consequence,  the mild formulation     \eqref{eq: mild} can    be rewritten as 
\[
u(t,x) = 1 +   \delta\Big(   p_{t-\cdot} (x-*) \sigma\big(  u(\cdot, \ast) \big)\Big).
\]
It is known   that for any  $(t,x)\in\RR_+\times \RR^d$,  $u(t,x)\in\mathbb{D}^{1,p}$ for any    $p\ge 2$ and the derivative satisfies, for $t\ge s$, the linear   equation 
\begin{align}  \notag
D_{s,y}u(t,x) &=   p_{t-s} (x-y)   \sigma(u(s,y)) \\  \label{ecu1}
&\quad +  \int_s^t \int_{\R^d}  p_{t-r} (x-z)  \Sigma(r,z) D_{s,y} u(r,z) W(dr,dz),
\end{align}
where $\Sigma(r,z)$ is an adapted process, bounded by   $L$.  
 If in addition $\sigma\in C^1(\R)$, then
$\Sigma(r,z)= \sigma'(u(r,z))$. This result is proved in    \cite[Proposition 2.4.4]{Nualart}  in the case of  the  stochastic heat equation with  Dirichlet boundary conditions on $[0,1]$ driven by a space-time white noise. Its proof can be easily adapted to our case, see     also  \cite{CHN18,NQ}.

\medskip

In the end of this section, we record a technical lemma, whose proof can be found in  \cite[Lemma 3.11]{CH2}.

 \begin{lemma}\label{lemma: iteration}
  For any $p\in[ 2,+\infty)$,  $0  \leq  s  \leq t  \le T $ and $x,y \in \R^d$, we have for   every $(s,y)\in[0,T]\times \R^d$,
 \begin{equation}\label{ecu3}
  \| D_{s,y} u(t,x) \|_p   \le  C p_{t-s} (x-y)
 \end{equation}
  for some constant $C= C_{T,p}$ that may depend on $T$ and $p$. 
  \end{lemma}
 Note that in \cite{DNZ18},   the $p$-norm of the Malliavin derivative $D_{s,y}u(t, x)$ can be bounded by a multiple of the fundamental solution
of the wave equation. So it is natural to expect that a similar estimate like \eqref{ecu3} {\it may} hold for a large family of stochastic partial differential equations.

\section{Proof of Theorem \ref{thm:TV-distance}} \label{sec:thm1}

We first state a lemma, which will be used below. 

\begin{lemma}   \label{lemma:integral}
Let $Z$ be a standard Gaussian  vector on $\RR^d$ and $\beta\in(0,d)$. Then
\begin{equation}\label{lemint1}
\sup_{s>0}\int_{\RR^d}p_s(x-y)|x|^{-\beta}dx= \sup_{s>0} \E\Big[ |y+\sqrt{s}Z|^{-\beta}\Big]\leq C |y|^{-\beta}\,,
\end{equation}
for some  constant $C$ that only depends on $d$ and $\beta$.
\end{lemma}

\begin{proof}
We want to show that 
$$
\int_{\RR^d}p_s(x-y)|x|^{-\beta} dx \leq C |y|^{-\beta}\,,
$$
where $C$ does not depend on $s$.  We first prove this for $s=1$. 

\begin{align*}
\int_{\RR^d} e^{-\frac{|x-y|^2}{2}}|x|^{-\beta} {dx} &=\int_{|x|\leq \frac{|y|}{2}}e^{-\frac{|x-y|^2}{2}}|x|^{-\beta} dx + \int_{|x|> \frac{|y|}{2}}e^{-\frac{|x-y|^2}{2}}|x|^{-\beta} dx\\
&\leq   \int_{|x|\leq \frac{|y|}{2}} e^{-\frac{|y|^2}{8}} |x|^{-\beta} dx + 2^\beta\int_{|x|\geq \frac{|y|}{2}} e^{-\frac{|x-y|^2}{2}} |y|^{-\beta}dx\\
&\leq  C_1 e^{-\frac{|y|^2}{8}} |y|^{-\beta+d} + C_2 |y|^{-\beta},
\end{align*}
{where the constants $C_1$ and $C_2$ only depend on $d,\beta$. So we have proved \eqref{lemint1} for $s=1$.}

Therefore, we have for general $s>0$,
\begin{align*}
\E\Big[ |y+\sqrt{s}Z|^{-\beta}\Big] = \E\Big[  \Big\vert \big(  ys^{-1/2} + Z\big) \sqrt{s} \Big\vert^{-\beta} \Big] \leq C  \vert ys^{-1/2} \vert^{-\beta} (\sqrt{s} )^{-\beta} = C | y |^{-\beta}.
\end{align*}
That is,   we have proved \eqref{lemint1} for general $s>0$. \end{proof}

 Now we begin with the estimate \eqref{n0} recalled below.

\begin{proposition}  \label{pro:covariance2}
With the notation  \eqref{NOTA1} and \eqref{NOTA2}, we have
$$
\lim_{R\to+\infty} R^{\beta -2d} \E[ G^2_R(t)] =   k_\beta \int_0^t  \eta^2(s)\, ds.
$$
\end{proposition}
\begin{proof}  We   define $
  \Psi (s, y)=\E\big[ \sigma (u(s,0))\sigma (u(s,y)) \big] 
 $.  Then
\begin{align*}
\E  [ G_R^2(t)] &=    \int_0^t \int_{\R^{2d}} \varphi_R(s,y)  \varphi_R(s,z)  \E\big[ \sigma (u(s,y))\sigma (u(s,z)) \big]   | y-z| ^{-\beta} dydzds \\
&= \int_0^t \int_{\R^{2d}} \varphi_R(s,\xi+z  )\varphi_R(s,z)    \Psi(s,\xi)  | \xi| ^{-\beta} d\xi dzds \,,\quad\text{\small by stationarity.}
\end{align*}

We claim that
 \begin{equation} \label{ecu11}
 \lim_{\vert\xi \vert\rightarrow +\infty}  \sup_{0\le s\le t } | \Psi(s,\xi)  -\eta^2(s)|=0.
 \end{equation}
To see \eqref{ecu11}, we apply a version of the Clark-Ocone formula for square integrable functionals of the noise $W$ and     we   can write
\[
  \sigma (u(s,y)) = \E[ \sigma (u(s,y))]
 + \int_0^s    \int_{\R^d}  \E\Big[  D_{r,\gamma}\big( \sigma (u(s,y)) \big) | \mathcal{F}_r\Big]  \, W(dr, d\gamma) \,.
\]
As a consequence,
$
 \E\big[ \sigma (u(s,y) )\sigma  (u(s,z) ) \big]  = \eta^2(s)+ T(s, y,z),
$
 where
\begin{align}
T(s, y,z)  \notag
&=\int_0^s    \int_{\R^{2d}}   \E\Big\{  \E\Big[  D_{r,\gamma}\big( \sigma (u(s,y)) \big) | \mathcal{F}_r\Big]  \E\Big[ D_{r,\lambda}\big( \sigma (u(s,z))\big) | \mathcal{F}_r\Big]  \Big\} \qquad\\  \label{fb2}
&\qquad\qquad\qquad\qquad\qquad\times  | \gamma-\lambda|^{-\beta} d\gamma d\lambda dr.
\end{align}
By the chain-rule (see \cite{Nualart}),  $ D_{r,\gamma}\big( \sigma (u(s,y)) \big) = \Sigma(s,y) D_{r,\gamma} u(s,y)$, with $\Sigma(\cdot,\ast)$  $\mathbb{F}$-adapted and  bounded by  $L$ ({in particular, $\Sigma(s,y) = \sigma'(u(s,y))$, when $\sigma$ is differentiable}
.) This implies,  
 using  (\ref{ecu3}),
 \begin{align} \notag
 &\quad  \Big\vert \E\Big\{  \E\Big[  D_{r,\gamma}\big( \sigma (u(s,y)) \big) | \mathcal{F}_r\Big]  \E\Big[ D_{r,\lambda}\big( \sigma (u(s,z))\big) | \mathcal{F}_r\Big]  \Big\}  \Big\vert \\  \label{fb1}
 &  \leq   L^2 \|D_{r,\gamma} u(s,y)\|_2   \|D_{r,\lambda} u(s,z)\|_2 \leq C p_{s-r}(\gamma-y) p_{s-r} (\lambda-z),
 \end{align}
 for some constant $C$. Therefore,  substituting (\ref{fb1}) into  (\ref{fb2}) and using the semigroup property in \eqref{eq35}, we can write 
  \begin{align} \nonumber
   \vert T(s,y,z)  \vert   &  \leq   C  \int_0^s    \int_{\R^{2d}} 
    p_{s-r}(\gamma-y) p_{s-r} (\lambda-z) 
  | \gamma-\lambda|^{-\beta} d\gamma d\lambda dr  \\
  &  = C\int_0^s \int_{\R^d} p_{2s-2r}(\gamma+z-y)|\gamma|^{-\beta} d\gamma dr \label{eq35}\\
  & \leq C |y-z|^{-\beta} \xrightarrow{\vert y-z\vert\to+\infty} 0  \,,   \label{exp1}
  \end{align}
  {by Lemma \ref{lemma:integral}}. 
   This completes the verification of   \eqref{ecu11}.

  \medskip

Let us continue our proof of Proposition \ref{pro:covariance2} by  proving  that  
\begin{align}\label{quan}
 R^{\beta-2d}   \int_0^t \int_{\R^{2d}}  \varphi_R(s,\xi+z  )\varphi_R(s,z)    \big[ \Psi(s,\xi) -  \eta^2(s)\big]    \vert \xi \vert ^{-\beta} d\xi dzds  \to 0 \,,
 \end{align}
 as $R\to+\infty$. Notice that
\begin{align}  \label{mm1}
\int_{\R^d}  \varphi_R(s,\xi+z  )\varphi_R(s,z) dz = 
\int_{  B_R^2}  p_{2(t-s)} (x_1-x_2-\xi) dx_1dx_2 \le CR^d\,.
\end{align} 
From  (\ref{ecu11}) we deduce that 
given any $\varepsilon > 0$, we find $K = K_\varepsilon > 0$ such that   $| \Psi(s, \xi) -\eta^2(s)| < \varepsilon$ for all $s\in [0,t]$ and $|\xi|\geq K$.  Now we divide the  integration domain  in (\ref{quan}) into two parts $\vert \xi \vert\leq K$ and $\vert \xi \vert > K$.

\medskip
\noindent
{\it  Case \rm (i):}    On the region $| \xi | \le K$,
since   $\Psi(s,y-z) - \eta^2(s) = T(s,y,z )$ is uniformly bounded,  using (\ref{mm1}),  we can write
\begin{align*}
&\quad R^{\beta-2d}  \int_0^t \int_{\RR^{2d}}  \varphi_R(s,\xi+z  )\varphi_R(s,z)  \big| \Psi(s,\xi) - \eta^2(s)  \big|  \mathbf{1}_{\{ \vert\xi\vert \leq K\}}     \vert \xi \vert ^{-\beta} \, d\xi dzds \\
&\leq C R^{\beta-2d} \int_0^t \int_{|\xi | \le K} \left(    \int_{\RR^d}  \varphi_R(s,\xi+z  )\varphi_R(s,z) \, dz \right) \vert \xi \vert ^{-\beta } \, d\xi ds\\
&\le  C R^{\beta-d} \int_0^t \int_{|\xi | \le K}   \vert \xi \vert ^{-\beta } \, d\xi ds \xrightarrow{R\to+\infty} 0 \,,\quad \text{because of $\beta <d$.}
\end{align*}

\medskip
\noindent
{\it  Case \rm (ii):}  On the region $[0,t]\times B_K^c$,  $| \Psi(s, \xi) -\eta^2(s)| < \varepsilon$. Thus, 
    \begin{align}
&  \quad R^{\beta-2d}  \int_0^t \int_{\RR^{2d}}  \varphi_R(s,\xi+z  )\varphi_R(s,z)  \big| \Psi(s,\xi) - \eta^2(s)  \big| \mathbf{1}_{\{ \vert\xi\vert > K\}}     \vert \xi \vert ^{-\beta} \, d\xi dzds \notag \\
&\leq  \varepsilon R^{\beta-2d}  \int_0^t \int_{\R^d} \int_{ B_R^2} p_{2(t-s)} (x_1-x_2-\xi)  |\xi |^{-\beta} dx_1dx_2 d\xi ds  \notag\\
&=  \varepsilon R^{\beta-2d}  \int_0^t \int_{ B_R^2} \E\left[ \left| x_1-x_2 - \sqrt{2s} Z \right| ^{-\beta} \right] dx_1dx_2 ds \notag\\
&=  \varepsilon  \int_0^t \int_{ B_1^2} \E\left[ \left| x_1-x_2 -R^{-1} \sqrt{2s} Z \right| ^{-\beta} \right] dx_1dx_2 ds \leq C \varepsilon t k_\beta \,, ~\text{ by Lemma \ref{lemma:integral},} \notag
\end{align}
with $Z$  a  standard Gaussian  random vector on $\R^d$.   This establishes the limit \eqref{quan}, since $\varepsilon > 0$ is arbitrary.

\medskip

Now, it suffices to show that
\begin{align}  
  R^{\beta-2d}   \int_0^t \eta^2(s)  \int_{\R^2}  \varphi_R(s,\xi+z  )\varphi_R(s,z)          \vert \xi \vert ^{-\beta} d\xi dzds  
\to k_\beta   \int_0^t \eta^2(s)ds. \label{quan2}
 \end{align}
 as $R\to+\infty$. This follows from   arguments similar to those used  above. In fact,   previous computations 
  imply that the left-hand side of \eqref{quan2} is equal to 
  \begin{align*}
   \int_0^t  \eta^2(s) \left(\int_{ B_1^2} \E\Big[ | x_1-x_2 -R^{-1} \sqrt{2s} Z  | ^{-\beta} \Big] dx_1dx_2\right) ds.
   \end{align*}
In view of {Lemma \ref{lemma:integral}} and dominated convergence theorem, we obtain the limit in  \eqref{quan2} and hence  the proof of the proposition is completed. \end{proof}

\begin{remark}\label{rem2} By using the same argument as in the   proof of Proposition   \ref{pro:covariance2}, we obtain an asymptotic  formula for $\E\big[ G_R(t_i) G_R(t_j) \big]$ with $t_i, t_j\in\R_+$, which is a useful ingredient for our proof of  Theorem \ref{thm:functional-CLT}.    Suppose  $t_i , t_j\in\R_+$,    we   write  
\begin{align*}
\E\big[ G_R(t_i) G_R(t_j) \big] 
    = \int_0^{t_i\wedge t_j} \int_{\R^{2d}} \varphi^{(i)}_R(s,y)  \varphi^{(j)}_R(s,z) \Psi(s,y-z)     | y-z| ^{-\beta} dydzds\,,
\end{align*}
 with $ \varphi^{(i)}_R(s,y) = \int_{B_R} p_{t_i-s}(x-y)\, dx$,  and  we  obtain
\begin{align*}
&\quad \lim_{R\to+\infty}  R^{\beta-2d}  \E\big[ G_R(t_i) G_R(t_j) \big]   \\
&= \lim_{R\to+\infty}  R^{\beta-2d}  \int_0^{t_i\wedge t_j} ds~ \eta^2(s)   \int_{\R^{2d}} \varphi^{(i)}_R(s,y)  \varphi^{(j)}_R(s,z)    | y-z| ^{-\beta} dydz \\
&=  k_\beta   \int_0^{t_i\wedge t_j}   \eta^2(s) \, ds \,.
\end{align*}
\end{remark}

\bigskip

We are now ready to present the proof of Theorem \ref{thm:TV-distance}.

\begin{proof}[Proof of Theorem \ref{thm:TV-distance}]         Recall from \eqref{FRT} that 
\[
F_R:=F_R(t) = \delta(v_R) \,,  ~ \text{with $v_R(s,y) = \sigma_R^{-1}\mathbf{1} _{[0,t]}(s) \varphi_R(s,y) \sigma (u(s,y)).$}  \]
Moreover,
$$
D_{s,y}F_R =\mathbf{1} _{[0,t]}(s) \frac{1}{\sigma_R}\int_{B_R} D_{s,y}u(t,x)dx \,,
$$
and   by \eqref{ecu1} and Fubini's theorem, we can write 
 \begin{align*}  
&\quad \int_{B_R} D_{s,y} u(t,x) \, dx  \\
&=  \varphi_R(s,y)  \sigma(u(s,y))     + \int_s^t \int_{\R^d}   \varphi_R(r,z)  \Sigma(r,z) D_{s,y} u(r,z) W(dr,dz).
\end{align*}
Therefore, we have the following decomposition
$
\langle DF_R,v_R\rangle_{\HH}  = B_1 +B_2,
$
with
\begin{align*}
B_1 &=    \sigma_R^{-2} \int_0^t \int_{\R^{2d}}  \varphi_R(s,y)\varphi_R(s, y')   \sigma \big(u(s,y)\big)  \sigma\big(u(s,y')\big)    |y-y'| ^{-\beta} \, dydy'ds
\intertext{and}
B_2 &=   \sigma_R^{-2}\int_0^t \int_{\R^{2d}} 
\left( \int_s^t \int_{\R^d} \varphi_R(r,z) \Sigma(r,z) D_{s,y} u(r,z) \, W(dr,dz) \right) \\
&\qquad  \qquad\qquad\qquad\qquad  \times \varphi_R(s,y') \sigma\big(u(s,y')\big) | y-y'| ^{-\beta}~ dydy' ds.
\end{align*}
By Proposition \ref{lem: dist},  $d_{\rm TV}(F_R, Z) \leq  2 \sqrt{ {\rm Var} \big[\langle DF_R,v_R\rangle_{\HH} \big] } \le  2\sqrt{2}(A_1+A_2),
 $
 with  
 \begin{align*}
  A_1& = \sigma_R^{-2}  \int_0^t   ds \Bigg(   \int_{\R^{4d}}  \varphi_R(s,y)\varphi_R(s,y') 
\varphi_R(s, \tilde{y})\varphi_R(s, \tilde{y}')   |y-y'| ^{-\beta}    \\
&\times |\tilde{y}-\tilde{y}'| ^{-\beta}      {\rm Cov} \Big[ \sigma \big(u(s,y)\big)  \sigma\big(u(s,y')\big) , \sigma \big(u(s,\tilde{y})\big)  \sigma\big(u(s,\tilde{y}')\big) \Big] ~ dy dy' d\tilde{y}d\tilde{y}'
 \Bigg)^{1/2} 
\intertext{and}
  A_2 &  =   \sigma_R^{-2} \int_0^t \Bigg(   \int_{\R^{6d}}  \int_s^t   \varphi_R(r,z)\varphi_R(r,\tilde{z} )  \varphi_R(s,y')   \varphi_R(s,\tilde{y}') \\
  &\qquad\quad\times
     \E  \Big\{    \Sigma(r,z) D_{s,y} u(r,z)  \Sigma(r,\tilde{z}) D_{s,\tilde{y}} u(r,\tilde{z}) \sigma\big(u(s,\tilde{y}')\big)\sigma\big(u(s,y')\big) \Big\}  \\
& \qquad\qquad  \times   | y-y'| ^{-\beta} | \tilde{y}-\tilde{y}'| ^{-\beta}  |z-\tilde{z}| ^{-\beta} dydy'  d\tilde{y} d\tilde{y}'   dz d\tilde{z} dr
   \Bigg)^{1/2} ds \,.
 \end{align*}
 The proof will be done in two steps:
 
 \medskip
 \noindent {\it Step 1:} ~
 Let us first estimate the term $A_2$.  By Lemma  \ref{lemma: iteration},
    \begin{align*}
&\quad  \Big\vert    \E  \Big\{    \Sigma(r,z) D_{s,y} u(r,z)  \Sigma(r,\tilde{z}) D_{s,\tilde{y}} u(r,\tilde{z}) \sigma\big(u(s,\tilde{y}')\big)\sigma\big(u(s,y')\big) \Big\} \Big\vert  \\
&   \le K_4^2(t) L^2 \| D_{s,y} u(r,z) \|_4 \| D_{s, \tilde{y}} u(r,\tilde{z}) \|_4 \leq C   p_{r-s} (y-z) p_{r-s}(\tilde{y} -\tilde{z}),
 \end{align*}
 where     $K_4(t)$ is  introduced in \eqref{Kp}.   By Proposition \ref{pro:covariance2}, we have 
          \begin{align*}
  A_2 &\leq  C R^{\beta-2d}  \int_0^t \Bigg(   \int_{\R^{6d}}  \int_s^t   \varphi_R(r,z)\varphi_R(r,\tilde{z} )  \varphi_R(s,y')   \varphi_R(s,\tilde{y}') 
     p_{r-s} (y-z) \\
&\qquad    \times p_{r-s}(\tilde{y} -\tilde{z})    | y-y'| ^{-\beta} | \tilde{y}-\tilde{y}'| ^{-\beta}  |z-\tilde{z}| ^{-\beta } dydy'  d\tilde{y} d\tilde{y}'   dz d\tilde{z} dr
   \Bigg)^{1/2} ds \,,
 \end{align*}
 where the  integral  in the spatial variables   can be rewritten as
           \begin{align*} 
 \mathbf{I}_R&:=      \int_{ B_R^4}    \int_{\R^{6d} }  p_{t-r} (x-z)  p_{t-r} (\tilde{x}-\tilde{z})  p_{t-s} (x'-y') p_{t-s} (\tilde{x}'-\tilde{y}') 
  p_{r-s} (y-z)  \\
&\qquad   \times  p_{r-s}(\tilde{y} -\tilde{z})  | y-y'| ^{-\beta} | \tilde{y}-\tilde{y}'| ^{-\beta}  |z-\tilde{z}| ^{-\beta} ~ dx d\tilde{x} dx'   d\tilde{x}'dydy'  d\tilde{y} d\tilde{y}'   dz d\tilde{z}\,. 
\end{align*}
Making the change of variables $\big(\theta= x-z$, $\tilde{\theta}= \tilde{x} -\tilde{z}$, $\eta'= x'-y'$, $\tilde{\eta}'= \tilde{x}' -\tilde{y}'$, $\eta=y-z$ and $\tilde{\eta} = \tilde{y} -\tilde{z}\big)$ yields,
\begin{align*}
& \mathbf{I}_R =      \int_{ B_R^4}   \Bigg(  \int_{\R^{6d} }  p_{t-r} (\theta)  p_{t-r} (\tilde{\theta})  p_{t-s} (\eta') p_{t-s} (\tilde{\eta}') 
  p_{r-s} (\eta) p_{r-s}(\tilde{\eta}) \\
&\qquad\qquad\qquad  \times   |  x-x' -\theta +\eta +\eta'| ^{-\beta} | \tilde{x}- \tilde{x}'  - \tilde{\theta} + \tilde{\eta} + \tilde{\eta}'| ^{-\beta}\\
&\qquad\qquad \qquad \qquad \times   |x-\tilde{x} -\theta +\tilde{\theta}| ^{-\beta} ~ d\eta d\eta'  d\tilde{\eta} d\tilde{\eta}'   d\theta d\tilde{\theta} \Bigg)dx d\tilde{x} dx'   d\tilde{x}'.
\end{align*}
This can be written as
\begin{align*}
& \mathbf{I}_R =  \int_{ B_R^4} dx d\tilde{x} dx'   d\tilde{x}' ~\E\Bigg\{     \big\vert  x-x' +\sqrt{t-r}Z_1 +\sqrt{r-s}Z_5 + \sqrt{t-s}Z_3\big\vert^{-\beta}                             \\
& \times \big\vert \tilde{x}- \tilde{x}'  +\sqrt{t-r}Z_2 + \sqrt{r-s}Z_6 + \sqrt{t-s}Z_4 \big\vert^{-\beta}  \big\vert x-\tilde{x} +\sqrt{t-r}(Z_1 + Z_2)\big\vert^{-\beta} \Bigg\}\\
&\leq C\int_{ B_R^4} dx d\tilde{x} dx'   d\tilde{x}' ~\E\Big\{     \big\vert  x-x' +\sqrt{t-r}Z_1 \big\vert^{-\beta}          \big\vert \tilde{x}- \tilde{x}'  -\sqrt{t-r}Z_2  \big\vert^{-\beta}                     \\
&\qquad\qquad\qquad \qquad \times \big\vert x-\tilde{x} +\sqrt{t-r}(Z_1 + Z_2)\big\vert^{-\beta} \Big\}, 
\end{align*}
with $Z_1, \ldots, Z_6$  i.i.d. standard Gaussian vectors on $\R^d$, {here we have used Lemma \ref{lemma:integral} {repeatedly  for $Z_5, Z_3, Z_6, Z_4$ to obtain the last inequality. }}
Making the change of variables  $x\to Rx$, $\tilde{x} \to  R\tilde{x} $, $ x' \to Rx'$ and $\tilde{x} ' \to R \tilde{x}'$, we can write
\begin{align*}
 \mathbf{I}_R& =  CR^{4d-3\beta}~ \E \int_{ B_1^4} dx d\tilde{x} dx'   d\tilde{x}'  ~  \big\vert  x-x' + R^{-1}\sqrt{t-r}Z_1  \big\vert^{-\beta}       \\
&\qquad \qquad\qquad \times \big\vert \tilde{x}- \tilde{x}'  -R^{-1}\sqrt{t-r}Z_2  \big\vert^{-\beta}  \big\vert x-\tilde{x} +R^{-1} \sqrt{t-r} (Z_1 + Z_2)\big\vert^{-\beta}  \\
& \quad \leq CR^{4d-3\beta}  \E\int_{ B_2^3} dy_1 dy_2   dy_3~   \big\vert  y_1 + R^{-1}\sqrt{t-r}Z_1  \big\vert^{-\beta}  \big\vert y_2  -R^{-1}\sqrt{t-r}Z_2  \big\vert^{-\beta} \\
&\qquad\qquad\qquad\qquad\quad \times \big\vert y_3+R^{-1} \sqrt{t-r} (Z_1 + Z_2)\big\vert^{-\beta} ,
\end{align*}
where the second inequality follows from the change of variables $x-x'=y_1$, $\tilde{x} -\tilde{x} ' =y_2$, $x-\tilde{x} =y_3$.
Taking into account the fact that
\begin{align}\label{easy2}
 \sup_{z\in\R^d}  \int_{ B_2}     \big\vert  y + z   \big\vert^{-\beta}  dy  < +\infty,
 \end{align}
 we obtain
 \[
 \mathbf{I}_R \leq CR^{4d-3\beta} \left( \sup_{z\in\R^d} \int_{B_2}\vert y+z \vert^{-\beta} dy \right)^3\leq CR^{4d-3\beta} \,,
\]
and it follows immediately that $   A_2 \le C R^{-\beta/2}   $.

\bigskip
 \noindent {\it Step 2:} ~ 
We now estimate $A_1$.   We begin by estimating the covariance
 \begin{align}\label{COVQ}
 {\rm Cov} \Big[ \sigma \big(u(s,y)\big)  \sigma\big(u(s,y')\big) , \sigma \big(u(s,\tilde{y})\big)  \sigma\big(u(s,\tilde{y}')\big) \Big]  \,.
   \end{align}
Using  a version of Clark-Ocone formula for square integrable functionals of the noise $W$, we can write
  \begin{align*}
\sigma \big(u(s,y)\big)  \sigma\big(u(s,y')\big) &=  \E\big[\sigma \big(u(s,y)\big)  \sigma\big(u(s,y')\big)\big] \\
&\quad+  \int_0^s \int_{\R^d} \E\Big[ D_{r,z} \Big(\sigma\big(u(s,y)\big)   \sigma\big(u(s,y')\big) \Big) | \mathcal{F}_r \Big] W(dr,dz).
  \end{align*}
  Then, we represent the covariance \eqref{COVQ} as 
       \begin{align*}
           & \int_0^s \int_{\R^{2d}} \E \Big\{ \E\Big[ D_{r,z} \Big(\sigma\big(u(s,y)\big)   \sigma\big(u(s,y')\big) \Big)| \mathcal{F}_r \Big] \\
        &\qquad\qquad\qquad \times \E\Big[ D_{r,z'} \Big( \sigma\big(u(s,\tilde{y})\big)   \sigma\big(u(s,\tilde{y}')\big)\Big) | \mathcal{F}_r \Big]  \Big\}         |z-z'|^{-\beta}  ~dz dz'dr .
                 \end{align*}
        By   the chain rule, 
                  $
                 D_{r,z}  \big(\sigma (u(s,y) )   \sigma (u(s,y') )\big)  = \sigma\big(u(s,y)\big) \Sigma(s,y') D_{r,z} u(s,y')
               + \sigma\big(u(s,y')\big) \Sigma(s,y) D_{r,z} u(s,y)$.
 Therefore, $\big\|  \E\big[ D_{r,z}  (\sigma (u(s,y) )   \sigma (u(s,y') )  )| \mathcal{F}_r \big]    \big\| _2$ is bounded by 
                 $
                2K_4(t) L \big\{   \| D_{r,z} u(s,y) \|_4+  \| D_{r,z} u(s,y') \|_4\big\}
              $.
Using   Lemma \ref{lemma: iteration} again,  we see that the covariance  \eqref{COVQ} is bounded by
        \begin{align*}
        &      4L^2 K_4^2(t)    \int_0^s \int_{\R^{2d} } \Big(  \left\| D_{r,z} u(s,y)\right\|_4 +\left\| D_{r,z} u(s,y')\right\|_4  \Big) \\
        &\qquad\quad \qquad\qquad \times   \Big(   \| D_{r,z'} u(s,\tilde{y}) \|_4 + \| D_{r,z'} u(s,\tilde{y}') \|_4  \Big)  |z-z'| ^{-\beta} dz dz'dr \\
        &   \le C  \int_0^s \int_{\R^{2d} }   \big( p_{s-r}(y-z)   +  p_{s-r}(y'-z) \big) \\
        & \qquad\qquad\qquad\qquad \times   \big( p_{s-r}(\tilde{y}-z')   +  p_{s-r}(\tilde{y}'-z') \big)  |z-z'| ^{-\beta} dz dz'dr . 
                 \end{align*}
Consequently, the spatial integral in the expression of $A_1$ can be bounded by 
\begin{align*}
  \mathbf{J}_R&:=C\int_0^s\int_{\R^{6d}}  \varphi_R(s,y)\varphi_R(s,y') 
\varphi_R(s, \tilde{y})\varphi_R(s, \tilde{y}')     \\
& \quad\quad\quad\times  |y-y'| ^{-\beta}  |\tilde{y}-\tilde{y}'| ^{-\beta}  |z-z'| ^{-\beta}    \Big( p_{s-r}(y-z)   +  p_{s-r}(y'-z) \Big) \\
        & \qquad\qquad\qquad\qquad \times   \Big( p_{s-r}(\tilde{y}-z')   +  p_{s-r}(\tilde{y}'-z') \Big)  dy dy' d\tilde{y}d\tilde{y}' dz dz'dr \\
        &= 4C\int_0^s dr \int_{\R^{6d}} dy dy' d\tilde{y}d\tilde{y}' dz dz' ~\varphi_R(s,y)\varphi_R(s,y') 
\varphi_R(s, \tilde{y})\varphi_R(s, \tilde{y}')     \\
& \quad   \times  |y-y'| ^{-\beta}  |\tilde{y}-\tilde{y}'| ^{-\beta}  |z-z'| ^{-\beta}    p_{s-r}(y-z)  p_{s-r}(\tilde{y}-z') ,   \quad\text{\small by  symmetry. }
        \end{align*}
 Then, it follows from exactly the same argument as in the estimation of $\mathbf{I}_R$ in  the previous step that
                $\mathbf{J}_R \leq C R^{4d-3\beta}$.    Indeed, we have 
  \begin{align}       
&\quad \int_0^s   \int_{B_R^4} \int_{\R^{6d}}   p_{t-s} (x-y) p_{t-s} (x'-y')p_{t-s} (\tilde{x}-\tilde{y})p_{t-s} (\tilde{x}'-\tilde{y}')   \notag  \\
& \quad\quad\quad\qquad\qquad\times  |y-y'| ^{-\beta}  |\tilde{y}-\tilde{y}'| ^{-\beta}  |z-z'| ^{-\beta}  p_{s-r}(y-z)  p_{s-r}(\tilde{y}-z') \notag\\
        & \qquad\qquad\qquad\qquad\qquad\qquad    dx dx' d\tilde{x} d\tilde{x}' dy dy' d\tilde{y}d\tilde{y}' dz dz'dr \notag \\
  &= \int_0^s dr \int_{B_R^4} dx dx' d\tilde{x} d\tilde{x}'   \int_{\R^{6d}}p_{t-s} (\theta) p_{t-s} (\theta')p_{t-s} (\tilde{\theta})p_{t-s} (\tilde{\theta}' )p_{s-r}(\eta) \notag  \\
        &\qquad\qquad\qquad\times p_{s-r}(\tilde{\eta})  ~\vert  x - x' -\theta  + \theta'\vert^{-\beta}    \vert \tilde{x} - \tilde{x}' -\tilde{\theta} + \tilde{\theta}'\vert^{-\beta}  \notag \\
      &\qquad\qquad\qquad\qquad\qquad  \times  \vert  x - \tilde{x}  -\eta + \theta + \tilde{\theta} + \tilde{\eta} \vert^{-\beta} d\theta d\theta' d\tilde{\theta}d\tilde{\theta}' d\eta d\tilde{\eta}  \notag\\
        &= \int_0^s dr \int_{B_R^4}  \E\Big\{  \big\vert  x - x'  + \sqrt{t-s}(Z_2-Z_1) \big\vert^{-\beta} \big\vert \tilde{x} - \tilde{x}'   +   \sqrt{t-s}(Z_4-Z_3) \big\vert^{-\beta}   \notag \\
        &\qquad \qquad \times       \big\vert  x - \tilde{x}  + \sqrt{t-s}(Z_1 + Z_3) + \sqrt{2s-2r}Z_5  \big\vert^{-\beta}  \Big\} ~dx dx' d\tilde{x} d\tilde{x}'  \notag\\
        &\leq C\int_0^s dr \int_{B_R^4}  \E\Big\{  \big\vert  x - x'  - \sqrt{t-s}Z_1 \big\vert^{-\beta} \big\vert \tilde{x} - \tilde{x}'   -   \sqrt{t-s}Z_3 \big\vert^{-\beta} \notag   \\
        &\qquad \qquad\qquad\qquad \times       \big\vert  x - \tilde{x}  + \sqrt{t-s}(Z_1 + Z_3)   \big\vert^{-\beta}  \Big\}~dx dx' d\tilde{x} d\tilde{x}'   \label{uselem} \\
                &\leq  CR^{4d-3\beta}~\E  \int_{B_1^4} dx dx' d\tilde{x} d\tilde{x}'    \big\vert  x - x'  -R^{-1} \sqrt{t-s} Z_1 \big\vert^{-\beta} \notag\\
       &\qquad\qquad \times  \big\vert \tilde{x} - \tilde{x}'   - R^{-1}  \sqrt{t-s}Z_3 \big\vert^{-\beta}    \times      \big\vert  x - \tilde{x}  +R^{-1} \sqrt{t-s}(Z_1 + Z_3) \big\vert^{-\beta} \notag \\
       &\leq CR^{4d-3\beta}~\E  \int_{B_2^3}  \big\vert  y_1  -R^{-1} \sqrt{t-s} Z_1 \big\vert^{-\beta}\notag \\
       &\qquad\quad \times  \big\vert y_2   - R^{-1}  \sqrt{t-s}Z_3 \big\vert^{-\beta}    \times      \big\vert y_3  +R^{-1} \sqrt{t-s}(Z_1 + Z_3) \big\vert^{-\beta}  dy_1 dy_2 dy_3  \notag\\
       &\leq  CR^{4d-3\beta} \left( \sup_{z\in\R^d}\int_{B_2}  \big\vert  y + z \big\vert^{-\beta} \, dy \right)^3  \leq CR^{4d-3\beta} \,,\quad\text{ by \eqref{easy2},}
\end{align}    
 where we have used Lemma \ref{lemma:integral} for $Z_5, Z_2, Z_4$ to obtain \eqref{uselem}.      
           This gives us the desired estimate for $A_1$ and finishes our proof.    \qedhere
                
            \end{proof}

 With a slight modification of the above proof, we can extend Theorem \ref{thm:TV-distance} to  more general initial conditions.
 
 \begin{corollary}
 \label{cor:TV-general-initial}
 Assume that there are two positive constants $c_1,c_2$ such that $c_1\leq u(0,x)\leq c_2$, and {assume one of the following two conditions:
 \begin{enumerate}
 \item $\sigma$ is a non-negative, nondecreasing Lipschitz  function with $\sigma(c_1) > 0$.
 
 \item  $-\sigma$  is a non-negative, nondecreasing Lipschitz  function with    $\sigma(c_1) < 0$.
 \end{enumerate}
} Define 
 \begin{equation}
 \widetilde{F}_R(t) = \frac{1}{\widetilde{\sigma}_R}  \int_{B_R}\Big[u(t,x)- \E u(t,x)  \Big]dx\,,
 \end{equation}
 where 
 $$
 \widetilde{\sigma}_R^2 = {\rm Var}\left( \int_{B_R}\Big[u(t,x)- \E u(t,x)  \Big]dx \right) {\in (0, \infty) }
 $$  {for every $R>0$. }         Then, there exists a constant $C = C(t,\beta)$, depending on $t$ and $\beta$, such that 
\[
d_{\rm TV}\left(  \widetilde{F}_R(t), ~N(0,1)\right) \leq C R^{-\beta/2} \,.
\]

 \end{corollary}
\begin{proof}
We write $\widetilde{F}_R(t)$ as $\widetilde{G}_R(t) / \widetilde{\sigma}_R$. Let $u_1(t,x)$ be the solution to  equation \eqref{eq:heat-equation} with initial condition $u(0,x)=c_1$. According to the weak comparison principle (see \cite[Theorem 1.1]{CH}), $u(t,x)\geq u_1(t,x)$ almost surely  for every $t\geq0$ and $x\in \RR^d$, which  immediately implies 
$$
\sigma\big(u(s,y)\big)\sigma\big(u(s,z)\big) \geq   \sigma\big(u_1(s,y)\big)\sigma\big(u_1(s,z)\big)  \geq 0  \quad\text{almost surely,}
$$   so that 
\begin{align}
&\E\big[ \widetilde{G}^2_R(t) \big]= \int_0^t \int_{\RR^{2d}} \varphi_R(s,y)\varphi_R(s,z)\E \Big[ \sigma\big(u(s,y)\big)\sigma\big(u(s,z)\big) \Big] |y-z|^{-\beta} dy dz ds \notag \\
& \geq \int_0^t \int_{\RR^{2d}} \varphi_R(s,y)\varphi_R(s,z)\E \Big[ \sigma\big(u_1(s,y)\big)\sigma\big(u_1(s,z)\big)  \Big] |y-z|^{-\beta} dy dz ds. \label{lasint}
\end{align}
In view of point (iii) from Remark \ref{rem0}, our assumption $\sigma(c_1)\neq 0$ guarantees that the last integral \eqref{lasint} has the exact order $R^{2d-\beta}$: More precisely,
\[
\eqref{lasint} \sim \kappa_\beta \left( \int_0^t \eta_1(s)^2 ds \right) R^{2d-\beta} \quad \text{as $R\to+\infty$,}
\]
where $\eta_1(s) := \E\big[ \sigma(u_1(s,y))\big]$ does not depend on $y$ and $\int_0^t \eta_1(s)^2 ds\in(0,\infty)$; see also \eqref{NOTA2}. 

\medskip

Now let $u_2(t,x)$ be the solution to the equation \eqref{eq:heat-equation} with initial condition $u(0,x)=c_2$. Notice that by our assumption on $\sigma$, we get $\sigma(c_2)\neq 0$. And by applying the same weak comparison principle, we deduce that 
\begin{align*}
\E\big[ \widetilde{G}^2_R(t) \big] & \leq \int_0^t \int_{\RR^{2d}} \varphi_R(s,y)\varphi_R(s,z)\E \Big[ \sigma\big(u_2(s,y)\big)\sigma\big(u_2(s,z)\big)  \Big] |y-z|^{-\beta} dy dz ds\\
&\sim \kappa_\beta \left( \int_0^t \eta_2(s)^2 ds \right) R^{2d-\beta} \quad \text{as $R\to+\infty$,}
\end{align*}
where $\eta_2(s) := \E\big[ \sigma(u_2(s,y))\big]$ does not depend on $y$ and $\int_0^t \eta_2(s)^2 ds\in(0,\infty)$. This implies that
\[
 0 < \liminf_{R\to+\infty}   \widetilde{\sigma}_R^2     R^{\beta-2d}    \leq  \limsup_{R\to+\infty}   \widetilde{\sigma}_R^2     R^{\beta-2d}  < +\infty\,.
\]
The rest of the proof follows    the same lines as the proof of Theorem \ref{thm:TV-distance}. 
\end{proof}

 \section{Proof of Theorem \ref{thm:functional-CLT}}
\label{sec:thm2}
In order to prove Theorem \ref{thm:functional-CLT}, we need to establish  the  convergence of the finite-dimensional distributions as well as the tightness.

\medskip

 \begin{proof}[\bf Convergence of finite-dimensional distributions] ~   Fix   $0\le t_1< \cdots <t_m \le T$ and consider  
\[
F_R(t_i)  :=   R^{\frac {\beta}2 -d} \int_{B_R} \big[ u(t_i,x) - 1 \big]~dx    = \delta\big( v_R^{(i)}\big) ~\text{for $i=1, \dots, m$,}
\]
where   
$$
v_R^{(i)} (s,y) =    \mathbf{1} _{[0,t_i]}(s) R^{\frac {\beta}2 -d}  \sigma\big(u(s,y)\big)\varphi^{(i)}_R(s,y)   
$$
with $   \varphi^{(i)}_R(s,y) =    \int_{B_R}   p_{t_i -s} (x-y)  dx$.
Set $\mathbf{F}_R=\big( F_R(t_1), \dots, F_R(t_m) \big)$ and let $Z$ be a centered  Gaussian   vector on $\RR^m$ with covariance $(C_{i,j})_{1\leq i,j\leq m}$ given by  
\[
C_{i,j}    :=  k_\beta  \int _0^{t_i \wedge t_j}     \eta^2(r)  ~dr \,,~ \text{with $\eta(r):= \E\big[\sigma\big(u(r,y)\big)\big]$.}
\]
To proceed,   we  need the following  generalization of    \cite[Theorem 6.1.2]{NP};  see \cite[Proposition 2.3]{HNV18} for a proof.

\begin{lemma}\label{lemma: NP 6.1.2}
Let $F=( F^{(1)}, \dots, F^{(m)})$ be a random vector such that $F^{(i)} = \delta (v^{(i)})$ for $v^{(i)} \in {\rm Dom}\, \delta$ and
$  F^{(i)} \in \mathbb{D}^{1,2}$, $i = 1,\dots, m$. Let $Z$ be an $m$-dimensional centered Gaussian  vector with covariance $(C_{i,j})_{ 1\leq i,j\leq m} $. For any  $C^2$ function $h: \R^m \rightarrow \R$ with bounded second partial derivatives, we have
\[
\big| \E [ h(F)] -\E [ h(Z)]  \big| \le \frac{m}{2}  \|h ''\|_\infty \sqrt{   \sum_{i,j=1}^m   \E \Big[ \big(C_{i,j} - \langle DF^{(i)}, v^{(j)} \rangle_{\HH}\big)^2 \Big] } \,,
\]
where $\|h ''\|_\infty : = \sup\big\{   \big\vert \frac{\partial^2}{\partial x_i\partial x_j} h(x) \big\vert\,:\, x\in\R^d\, , \, i,j=1, \ldots, m \big\}$.
\end{lemma}
 
 \noindent In view of  Lemma \ref{lemma: NP 6.1.2}, the proof of   $\mathbf{F}_R\xrightarrow{\rm law} Z$ boils down to   the $L^2(\Omega)$-convergence of $\langle DF_R(t_i), v_R^{(j)} \rangle_{\HH}$  to  $C_{i,j} $ for each $i,j$, as $R\to+\infty$.  The case $i=j$ has been covered in the proof of Theorem \ref{thm:TV-distance} and for the other cases,       we need to show
\begin{equation} \label{bb1}
 \E\big[  F_R(t_i) F_R(t_j)  \big] \to C_{i,j}
 \end{equation}
 and
 \begin{equation} \label{bb2}
 {\rm Var}\big( \langle DF_R(t_i), v_R^{(j)} \rangle_{\HH} \big)\to 0,  
 \end{equation}
  as $R\to+\infty$.
The point (\ref{bb1})   has been established in {\it Remark} \ref{rem2}.  To see point (\ref{bb2}),  we put
         \begin{align*}
         B_1(i,j)        & =   R^{\beta -2d}  \int_0^{t_i\wedge t_j} \int_{\R^d}  \varphi_R^{(i)}(s,y) \varphi_R^{(j)}(s,y')   \sigma (u(s,y)) \sigma(u(s,y'))\\
         & \qquad \times  |y-y'|^{-\beta}  dy dy' ds 
\end{align*}
and
\begin{align*}
B_2(i,j)&=  R^{\beta -2d}  \int_0^{t_i\wedge t_j} \int_{\R^{2d}} \varphi^{(j)}_R(s,y') \sigma(u(s,y'))|y-y'|^{-\beta}  \\
&\qquad  \times\left( \int_s^{t_i}\int_{\R^d} \varphi^{(i)}_R(r,z) \Sigma(r,z) D_{s,y} u(r,z) W(dr,dz) \right)  \,    dydy'ds,
\end{align*}
so that $\langle DF_R(t_i), v_R^{(j)} \rangle_{\HH}  = B_1(i,j) +B_2(i,j)  $. Therefore, 
\begin{align*}
& \quad \E \left[\Big(C_{ij}-\langle DF_R(t_i), v_R^{(j)}\rangle_{\mathfrak{H}} \Big)^2\right] \\
& \leq 3 \big(C_{i,j} - \E[ B_1(i,j) ]  \big)^2  +  3  {\rm Var}\big[ B_1(i,j) \big]  + 3  {\rm Var} \big[ B_2(i,j)\big]\,.
\end{align*}
Going through the same lines as for the estimation of $A_1, A_2$, one can verify  easily that both $ {\rm Var}\big[ B_1(i,j) \big]$ and $  {\rm Var} \big[ B_2(i,j)\big]$ vanish asymptotically, as $R$ tends to infinity. By recalling that $C_{ij} = \lim_{R \to +\infty} \E \big[ B_1(i,j)\big]$, the convergence of finite-dimensional distributions is thus established. \qedhere

\end{proof}

\begin{proof}[\bf Tightness]      The following proposition together with Kolmogorov's criterion ensures the tightness of the processes
$\big\{  R^{\frac {\beta} 2-d}  \int_{B_R} \big[ u(t,x) -1 \big] \,dx \,, t\in[0,T]   \big\}$, $R > 0$.

\medskip

\begin{proposition}
\label{pro:tightness}
Recall the notation \eqref{NOTA1}. For any $0\leq s < t\leq T$,  for any $p\in[ 2,\infty)$ and any $ \alpha \in(0, 1)$, it holds that 
\begin{align}\label{tendue}
\E\big[ \vert G_R(t) - G_R(s) \vert^p\big] \leq C R^{p(d-\frac  \beta 2)}(t-s)^{\alpha p/2}~ ,
\end{align}
where  the constant $C = C_{p,T,\alpha}$ may  depend on $T$, $p$ and $\alpha$.
\end{proposition}

Notice that   the $p$th  moment   of an increment of the solution $u(t,x) -u(s,x)$ is bounded by
 a constant times $|t-s|^ {\frac{\alpha p }{2}(1-\frac \beta2)}$, for any  $ \alpha \in(0, 1)$, see \cite{SS}; and here the 
 spatial  averaging has improved the H\"older continuity.

\medskip
Now we present the proof of   Proposition \ref{pro:tightness}.
  Let  $0\leq s < t\leq T$ and set 
$$
\Theta_{x,t,s}(r,y) = p_{t-r}(x-y)\textbf{1}_{\{r\leq t\}}-p_{s-r}(x-y)\textbf{1}_{\{r\leq s\}}.
$$
Then, 
$$ 
G_R(t) - G_R(s)= \int_0^T\int_{\R^d}\left(\int_{B_R} \Theta_{x,t,s}(r,y)\, dx\right)\sigma\big(u(r,y)\big) \,W(dr,dy) \,,
$$
so that 
\begin{align*}
&\quad \E\Big[  | G_R(t) - G_R(s)|^p \Big]\\
&\leq C~  \Bigg\| ~ \int_0^T dr \int_{\R^{2d}} dy dy'~ |y-y'|^{-\beta}  \left(\int_{B_R}\Theta_{x,t,s}(r,y)~dx\right) \sigma\big(u(r,y)\big)  \\
&\qquad\qquad  \times \left[ \int_{B_R}\Theta_{\tilde{x},t,s}(r,y) ~ d\tilde{x} \right] \sigma\big(u(r,y')\big)    \Bigg\|^{p/2}_{p/2} \quad \text{\small by Burkholder's inequality}\\
 &\leq  C \Bigg\{     \int_0^T dr \int_{\R^{2d}} dy dy' |y-y'|^{-\beta}  \left|\int_{B_R}\Theta_{x,t,s}(r,y)dx\right| \times \left|\int_{B_R}\Theta_{\tilde{x},t,s}(r,y)d\tilde{x}\right|\\
&\qquad\qquad \quad \times\big\|\sigma(u(r,y))\sigma(u(r,y'))\big\|_{p/2} \Bigg\}^{p/2} \quad \text{by Minkowski's inequality}.
\end{align*}
Therefore, taking into account that $ \|\sigma(u(r,y))\sigma(u(r,y'))\|_{p/2}$ is uniformly bounded, we obtain
\begin{align*}
&\quad \E\Big[  | G_R(t) - G_R(s)|^p \Big]\\
  &\leq C \left\{ ~ \int_0^T  \int_{\R^{2d}} \left\vert \int_{B_R}\Theta_{x,t,s}(r,y)dx\right\vert  \left|\int_{B_R}\Theta_{\tilde{x},t,s}(r,y)d\tilde{x}\right||y-y'|^{-\beta} ~ dy dy'dr\right\}^{p/2}\\
  &\leq C \Bigg\{ \int_0^s \int_{\R^{2d}} \left\vert\int_{B_R}\big(p_{t-r}(x-y) -p_{s-r}(x-y) \big) dx\right\vert\\
&\qquad\qquad \times\left\vert \int_{B_R}\big(p_{t-r}(\tilde{x}-y') -p_{s-r}(\tilde{x}-y') \big) d\tilde{x}\right\vert   \big\vert y-y' \big\vert^{-\beta}~dy dy' dr\Bigg\}^{p/2}\\
&\quad + C \Bigg[ \int_s^t \int_{\R^{2d}} \left(\int_{B_R^2}p_{t-r}(x-y)p_{t-r}(\tilde{x}-y')~dxd\tilde{x}\right) |y-y'|^{-\beta} dy dy' dr\Bigg]^{p/2}\\
&=:C \big(\mathbf{T}_1^{p/2} + \mathbf{T}_2^{p/2} \big)\,.
\end{align*}

\medskip

\noindent{\it Estimation of the  term $\mathbf{T}_1$.} We need  Lemma 3.1 in \cite{CH}: For all $\alpha\in(0,1)$, $x,y\in\R^d$ and $t'\ge t>0$, 
$$
 \left|p_t(x)-p_{t'}(x)\right| \leq C t^{-\alpha/2}(t'-t)^{\alpha/2}p_{4t'}(x)\,.$$ Thus,  
\begin{align*}
  &\qquad (t-s)^{-\alpha} \mathbf{T}_1\\
  & \leq   \int_0^s dr  (s-r)^{-\alpha} \int_{B_R^2}  dxd\tilde{x}  \int_{\R^{2d}}dy dy'     p_{4(t-r)}(x-y)     p_{4(t-r)}(\tilde{x}-y')     |y-y'|^{-\beta} \\
&=   \int_0^s dr  (s-r)^{-\alpha} \int_{B_R^2}  dxd\tilde{x}  \int_{\R^{2d}}dy dy'     p_{4(t-r)}(\theta)   p_{4(t-r)}(\tilde{\theta})     \big|x- \tilde{x} -\theta + \tilde{\theta} \big|^{-\beta} \,,
\end{align*}
by the change of variable $\theta = x - y$, $\tilde{\theta} = \tilde{x} - y'$. Consequently,  with $Z$ a standard Gaussian random vector on $\R^d$, we continue to write 
\begin{align*}
(t-s)^{-\alpha} \mathbf{T}_1&\leq    \int_0^s dr  (s-r)^{-\alpha} \int_{B_R^2}  dxd\tilde{x} ~\E\Big[ \vert x - \tilde{x} + \sqrt{8t-8r}Z\vert^{-\beta} \Big]  \\
&= R^{2d-\beta}  \int_0^s dr  (s-r)^{-\alpha} \int_{B_1^2}  dxd\tilde{x} ~\E\Big[ \vert x - \tilde{x} + R^{-1}\sqrt{8t-8r}Z\vert^{-\beta} \Big] \\
&\leq CR^{2d-\beta}   \int_0^s dr  (s-r)^{-\alpha} \int_{B_1^2}  dxd\tilde{x} \vert x - \tilde{x}  \vert^{-\beta} \quad \text{using Lemma \ref{lemma:integral}} \\
&\leq CR^{2d-\beta}   \,.
\end{align*}
  
  \bigskip 
  
\noindent{\it Estimation of the  term $\mathbf{T}_2$.}  \quad  We use a similar change of variable as before:
\begin{align*}
\mathbf{T}_2 &=  \int_s^t  dr \int_{B_R^2} dx d\tilde{x} \int_{\R^{2d}} p_{t-r}(\theta)p_{t-r}(\tilde{\theta} ) \big\vert x - \tilde{x} + \tilde{\theta} -\theta \big\vert^{-\beta} d\theta d\tilde{\theta}   \\
&= \int_s^t  dr \int_{B_R^2} dx d\tilde{x} ~\E\Big\{ \big\vert x - \tilde{x} + \sqrt{2t-2r} Z \big\vert^{-\beta} \Big\}  \\
&\leq CR^{2d-\beta}  \int_s^t  dr \int_{B_1^2} dx d\tilde{x}  \big\vert x - \tilde{x}  \big\vert^{-\beta} \quad \text{using Lemma \ref{lemma:integral}} \\
&= CR^{2d-\beta} (t-s) k_\beta \,.
\end{align*}
Combining the above two estimates, we obtain \eqref{tendue}. \qedhere

\end{proof}

\section{Acknowledgement}
The authors thank the anonymous referee for many constructive advices that improved this paper.

\end{document}